\documentclass[12pt]{amsart}
\usepackage{amssymb}
\usepackage{verbatim}

\begin{document}

\newtheorem{theorem}{Theorem}[section]
\newtheorem{lemma}[theorem]{Lemma}
\newtheorem{corollary}[theorem]{Corollary}
\newtheorem{proposition}[theorem]{Proposition}

\newtheorem{maintheorem}{Main Theorem}
\def\themaintheorem{\unskip}

\theoremstyle{definition}
\newtheorem{remark}{Remark}
\def\theremark{\unskip}
\newtheorem{scholium}{Scholium}
\newtheorem{claim}[remark]{Claim}
\newtheorem{definition}{Definition}
\newtheorem{problem}{Problem}[section]

\numberwithin{equation}{section}

\def\Re{\operatorname{Re} }
\def\Im{\operatorname{Im} }
\def\distance{\operatorname{distance\,} }
\def\domain{\operatorname{Domain\,} }
\def\e{\varepsilon}
\def\eps{\varepsilon}
\def\p{\partial}

\def\reals{ {{\mathbb R}} }
\def\complex{ {{\mathbb C}} }
\def\torus{ {{\mathbb T}} }
\def\naturals{ {{\mathbb N}} }
\def\integers{ {{\mathbb Z}} }
\def\complex{{\mathbb C}}

\def\scriptd{ {\mathcal D} }
\def\scripth{ {\mathcal H} }
\def\scriptr{ {\mathcal R} }
\def\scriptc{ {\mathcal C} }
\def\scripte{ {\mathcal E} }

\def\scriptb{ {\mathcal B} }
\def\scriptt{{\mathcal T}}
\def\scriptf{{\mathcal F}}
\def\scriptg{{\mathcal G}}
\def\scriptv{{\mathcal V}}
\def\scriptl{{\mathcal L}}
\def\scriptn{{\mathcal N}}
\def\scriptm{{\mathcal M}}
\def\frakF{{\mathfrak F}}

\newcommand\dbarl{\overline{D}_\lambda}
\newcommand\dbarlnu{\overline{D}_{\lambda_\nu}}
\newcommand{\Dbarl}[1]{\overline{D}_{\lambda,#1}}
\newcommand\dbarlstar{ {\overline{D}^*_{\lambda}} }
\newcommand\dbarlstarnu{{{\overline{D}}_{\lambda_\nu}^*}}
\newcommand{\Dbarlstar}[1]{\overline{D}_{\lambda,#1}^*}
\newcommand{\xdbarl}[1]{\overline{{\mathcal D}}_{\lambda,#1}}
\newcommand{\xDbarl}[2]{\overline{{\mathcal D}}_{\lambda,#1,#2}}
\newcommand{\xdbarlstar}[1]{\overline{{\mathcal D}}_{\lambda,#1}^*}
\newcommand{\xDbarlstar}[2]{\overline{\scriptd}_{\lambda,#1,#2}^*}
\def\diver{\operatorname{div}}

\def\dbarb{\bar\partial_b}
\def\dbar{\bar\partial}
\def\op{\operatorname{op}}
\def\lt{L^2}
\newcommand{\norm}[1]{ \|  #1 \| }
\newcommand{\Norm}[1]{ \Big\|  #1 \Big\| }
\newcommand{\set}[1]{ \left\{ #1 \right\} }
\def\one{{\mathbf 1}}

\author{Michael Christ}
\address{
        Michael Christ\\
        Department of Mathematics\\
        University of California \\
        Berkeley, CA 94720-3840, USA}
\email{mchrist@math.berkeley.edu}
\thanks{The author was supported in part by NSF grant DMS-0901569.}

\date{August 14, 2013.}

\title[Off-diagonal Decay of Bergman kernels]
{Off-diagonal Decay of Bergman kernels: \\ On a Conjecture of Zelditch}


\maketitle

\section{Introduction}

This paper is a preliminary study of an inverse problem concerning asymptotic behavior of Bergman kernels. 
Let $X$ be a compact complex manifold, without boundary. Let $X$ be equipped with a $C^\infty$ Hermitian metric $g$,
along with the metrics on the bundles $B^{(p,q)}(X)$ of forms of bidegree $(p,q)$ induced by $g$,
and the volume form on $X$ associated to the induced Riemannian metric.
Denote by $\rho(z,z')$ the Riemannian distance from $z\in X$ to $z'\in X$.

Let $L$ be a positive holomorphic line bundle over $X$.  Let $L$ be equipped with a $C^\infty$ Hermitian metric $\phi$ 
whose curvature form is positive at every point. 

For each positive integer $\lambda$, let the line bundle $L^\lambda$ be the tensor product of $\lambda$ copies of $L$.
$L^\lambda$ inherits from $\phi$ a Hermitian metric so that for any $v\in L_z$, 
its $\lambda$--fold tensor product satisfies $|v\otimes v\otimes \cdots\otimes v| = |v|^\lambda$.

Let $L^2(X,L^\lambda)$ be the Hilbert space of equivalence classes of all square integrable Lebesgue measurable
sections of $L^\lambda$. 
Let $H^2(X,L^\lambda)$ be the closed subspace of $L^2(X,L^\lambda)$ consisting of all holomorphic sections.  
The Bergman projection operator $B_\lambda$ is by definition the orthogonal projection from $L^2(X,L^\lambda)$ 
onto $H^2(X,L^\lambda)$.
The Bergman kernel $B_\lambda(z,z')$ is the associated distribution-kernel; 
$B_\lambda(z,z')$ is a complex linear endomorphism from the fiber $L^\lambda_{z'}$ to the fiber $L^\lambda_z$.

The asymptotic behavior of the Bergman kernels as $\lambda\to\infty$, 
on and within distance $O(\lambda^{-1/2})$ of the diagonal, has been intensively studied.
This paper is concerned instead with the large $\lambda$ behavior at a positive distance from the diagonal.
It is well known that $B_\lambda$ tends rapidly to zero as $\lambda\to\infty$, away from the diagonal.
For real analytic $g,\phi$ there is decay at an exponential rate: provided that $\rho(z,z')\ge\delta>0$,
\begin{equation} \label{eq:expdecay} B_\lambda(z,z')  = O(e^{-c\lambda})\end{equation} for some $c=c(\delta)>0$.
For $C^\infty$ metrics $g,\phi$, 
\begin{equation}  B_\lambda(z,z')  = O(e^{-A\sqrt{\lambda\log\lambda}})\end{equation} 
for all $A<\infty$ \cite{christberg}.
This rate of decay is optimal \cite{christcounter}; if $h(\lambda)\to\infty$ as $\lambda\to\infty$
then there exist $X,L,\phi,g$ with $\phi,g\in C^\infty$ and points  $z\ne z'$ such that
\begin{equation} \label{subexpdecay} \limsup_{\lambda\to\infty} \sup_{\rho(z,z')
\ge\delta} e^{h(\lambda)\sqrt{\lambda\log\lambda}}\,{|B_\lambda(z,z')|}=\infty.  \end{equation}

Zelditch \cite{zelditchconjecture} has advanced a bold conjecture in the inverse direction,
proposing that exponential decay \eqref{eq:expdecay} holds only when $\phi$
is real analytic, and moreover that exponential decay for even an arbitrarily 
sparse sequence of values of $\lambda$ tending to infinity implies analyticity.
This note provides evidence in support of the conjecture,
by establishing it within the same framework in which the examples of subexponential decay 
\eqref{subexpdecay} were constructed \cite{christcounter}, generalized to arbitrary dimensions.
This framework, in which partial symmetry makes possible a separation of variables which simplifies
analytic issues, has been a very fruitful source of counterexamples and illustrative examples in the theory of
the Bergman projections and the Cauchy-Riemann equations, and more generally in the theory of partial differential
operators, including those with multiple characteristics.
Therefore we consider the validity of the conjecture within this limited
framework to be compelling evidence in its favor.
Moreover, although some elements of the analysis given here do not extend in a straightforward way to the general case,
it is hoped that its overall structure and some of its elements will be useful guides.  

\section{The framework}
Consider $\complex^d$, with coordinates $z=(z_1,\cdots,z_d)$.
Write $z_j=x_j+iy_j$ and $x=(x_1,\cdots,x_d)$, $y=(y_1,\cdots,y_d)\in\reals^d$.
Write $z=x+iy\in \reals^d+i\reals^d$.

Let $X$ be the noncompact complex manifold $X=\complex^d$, and let $L$ be the trivial
line bundle $L=\complex^d\times\complex^1$. $B^{(0,1)}$ denotes the bundle of forms of
bidegree $(0,1)$ over $\complex^d$.
$X$ is equipped with its usual flat metric as a complex Euclidean space. Thus integration
will be performed with respect to Lebesgue measure, and $B^{(0,1)}$
is equipped with the usual metric under which $|\bar\omega_J|=1$ where $\bar\omega_J
= \bar\omega_{j_1}\wedge\cdots\wedge\bar\omega_{j_q}$, whenever $j_1<j_2<\cdots<j_q$.

The metric $\phi$ on $L$ is represented by a $C^\infty$ real-valued function 
$\phi(z)$. The norm of an element $(z,t)\in \complex^d\times\complex=L$ is $e^{-\phi(z)}|t|$;
the norm of an element $(z,t)\in \complex^d\times\complex\leftrightarrow L^\lambda$ is $e^{-\lambda\phi(z)}|t|$;
$L^2(X,L^\lambda)$ is the Hilbert space of all Lebesgue measurable functions
$f:\complex^d\to\complex$ that satisfy \[\norm{f}_{L^2(X,L^\lambda)}^2
= \int_{\complex^d} |f(z)|^2 e^{-2\lambda\phi(z)}\,dm(z)<\infty.\]

The essential feature of the framework under discussion here is that
$\phi(x+iy)$ is a function of $x$ alone. We therefore write $\phi\equiv \phi(x)$.
We assume that the curvature form of $\phi$ is strictly positive, and uniformly bounded above and below.
Thus there exists $C\in(0,\infty)$ such that for all $z\in\complex^d$ and all $v\in\complex^d$,
\begin{equation} C^{-1}|v|^2
\le \sum_{j,k=1}^d \frac{\partial^2\phi}{\partial z_j\partial \bar z_k}(z)\,v_j\,\bar v_k \le C|v|^2
\end{equation} 
Because of our symmetry assumption, this simplifies to
\begin{equation} C^{-1}|v|^2
\le \sum_{j,k=1}^d \frac{\partial^2\phi}{\partial x_j\partial x_k}(x)\,v_j\, v_k
\le C|v|^2 \ \ \text{for all} \ \ x\in\reals^d\ \text{and}\ v\in\reals^d.  \end{equation} 

Under this positivity assumption, 
the space $H^2(X,L^\lambda)$ of all entire holomorphic functions satisfying
$\iint_{\complex^d} |f(x+iy)|^2 e^{-2\lambda\phi(x)}\,dx\,dy<\infty$ is a closed subspace of the
space $L^2(X,L^\lambda)$ of all equivalence classes of Lebesgue measurable functions 
for which the same integral is finite. 
The Bergman kernel $B_\lambda$ represents the orthogonal projection of $L^2(X,L^\lambda)$ onto $H^2(X,L^\lambda)$.
$B_\lambda(z,z')$ is a $C^\infty$ function off of the diagonal for all $\lambda>0$. 
These objects are well-defined for all $\lambda\in(0,\infty)$; one need not restrict to integer values.

\begin{theorem}
Let $X,L,\phi$ be as described.
Let $U\subset\complex^d$ be an open set, and suppose that 
for each $\delta>0$ there exist a sequence $\lambda_\nu$ tending to $\infty$
and $c>0$ such that for all $(z,z')\in U\times U$
satisfying $|z-z'|\ge\delta$ and for all sufficiently large $\nu$, 
\begin{equation} |B_{\lambda_\nu}(z,z')|\le e^{-c\lambda_\nu}.\end{equation}
Then $\phi\in C^\omega(U)$.
\end{theorem}
To put it another way, the function $x\mapsto\phi(x)$ is real analytic on the projection of $U$ onto $\reals^d$.

\section{Outline}

The proof is by contradiction.  Two constructions mediate between the Bergman projections and the metric $\phi$. 
Firstly, for each $\lambda$ we consider certain holomorphic sections of the associated line bundle 
which depend holomorphically on an external parameter $\xi\in\complex^d$. Secondly, for each $\lambda$
this family of sections gives rise in turn to a scalar-valued holomorphic function $\xi\mapsto\scriptf_\lambda(\xi)$.

Complex zeroes of $\scriptf_\lambda$ having suitably small imaginary parts were the key to the 
construction in \cite{christcounter} of metrics $\phi$ for which the Bergman kernels decay slowly 
as $\lambda\to\infty$.  Here we show that conversely, exponential decay not only precludes such zeros, 
but also precludes exponentially small values of $\scriptf_\lambda$. This lack of small values is equivalent to 
an upper bound on $\big|\lambda^{-1} \log|\scriptf_\lambda|\,\big|$ which is uniform in $\lambda$.

A normal families argument gives a sequence $\lambda_\nu\to\infty$ for which these functions
converge; the limiting function is real analytic in a complex disk. 
The restriction of this limiting function to the real domain is calculated directly
as the composition of the Legendre transform of $\phi$ with $(\nabla\phi)^{-1}$.  
The final step is to show that analyticity of this composed transform implies analyticity of $\phi$.

The construction
of \cite{christcounter} was executed only in the lowest-dimensional case $d=1$, but here we investigate
matters in arbitrary dimensions. Two new issues arise in higher dimensions. 
Firstly, while the formula defining $\scriptf_\lambda$ extends straightforwardly, 
its interpretation is not immediately clear. Secondly, in order to obtain auxiliary functions 
with suitable growth properties needed to conclude 
that $\scriptf_\lambda$ cannot take on any exponentially small values,
we are led to solve the divergence equation $\diver(u)=f$ in $\reals^d$, in Hilbert spaces with 
weighted $L^2$ norm bounds.  A seemingly essential issue is that the analysis requires 
bounds with respect to weights $e^{-\Phi}$ where $\Phi$ is subconvex, 
which is the real analogue of plurisuperharmonicity, 
rather than of the standard plurisubharmonicity of $\dbar$ theory. 
This is quite different from the usual situation; indeed the equation cannot be solved with
satisfactory bounds for arbitrary (closed) data. The necessary condition for solvability 
of $\diver(u)=f$ with the specific $f$ that arises in our analysis,
turns out to be the (exponentially near) vanishing of $\scriptf_\lambda$ --- so that 
the vanishing of $\scriptf_\lambda$ gains an interpretation as the key to resolving the second main obstacle.

That this necessary condition is sufficient for solvability with suitable bounds for subconvex weights, 
cannot be established by the 
standard integration-by-parts analysis. We expend some effort to establish solvability with the desired bounds. 
The solution is used to construct holomorphic sections which, if the Bergman kernels decay exponentially,
are very nearly in the range of the adjoint operator $\dbar^*$ and consequently are nearly orthogonal
to themselves.  This is the desired contradiction.

\section{Notations and framework}
Variables in $\complex^d$ will often be denoted by $z=x+iy$ where $x,y\in\reals^d$.
For $z,w\in\complex^d$ we will write 
\begin{equation} z\cdot w = \sum_{j=1}^d z_jw_j,\end{equation}
with no complex conjugation.

Lebesgue measure on either $\complex^d$ or $\reals^d$ will be denoted by $m$.
$L^2(\complex^d,w)$ is the Hilbert space of all equivalence classes of Lebesgue measurable
scalar-valued functions with norm squared $\int_{\complex^d} |f(z)|^2 w(z)\,dm(z)$.
The same notation $L^2(\complex^d,w)$ is also used to
denote the Hilbert space of all equivalence classes of Lebesgue measurable
$(0,1)$ forms with norm squared $\int_{\complex^d} |f(z)|^2 w(z)\,dm(z)$,
where $|\sum_{j=1}^d f_j(z)\bar z_j|^2 = \sum_{j=1}^d |f_j(z)|^2$.
The Bergman projections $B_\lambda$ associated to the weights $e^{-2\lambda\phi}$
are the orthogonal projections from $L^2(\complex^d,e^{-2\lambda\phi})$
to its closed subspace of entire holomorphic functions.

The following hypotheses concerning $\phi:\complex^d\to\reals$ will be in force throughout the paper.
\begin{gather}
\text{$\phi\in C^\infty$.} 
\label{phiH1}
\\
\text{$\phi(x+iy)$ depends on $x$ alone.} 
\label{phiH2}
\\
\text{$\phi$ is strictly convex.}
\label{phiH3}
\\
\left(\frac{\partial^2\phi}{\partial {x_i}\partial {x_j}}\right)_{i,j=1}^d \text{ is comparable to the identity matrix,}
\label{phiH4}
\end{gather}
uniformly as a function of $x\in\reals^d$. 
We will abuse notation by using the symbol $\phi$ to denote two functions,
one with domain $\complex^d$ and one with domain $\reals^d$, related by $\phi(x+iy)  = \phi(x)$. 
It will be clear from the context which of the two is intended.

The Cauchy-Riemann operator $\dbar$, mapping scalar-valued functions to $(0,1)$--forms, is defined by
\begin{equation} \dbar f = \sum_{k=1}^d \frac{\partial f}{\partial {\bar z_k}} \,d\bar z_k. \end{equation}
where
\begin{equation} \frac{\partial}{\partial\bar z_k} = \frac{\partial}{\partial x_k} +i \frac{\partial}{\partial y_k}.  \end{equation}
We also write
\begin{equation} \partial_{z_k}
= \frac{\partial}{\partial z_k} = \frac{\partial}{\partial x_k} -i \frac{\partial}{\partial y_k}.  \end{equation}
For $w=e^{-\lambda\phi}$, the formal adjoint $\dbar^*_{2\lambda\phi}$  of $\dbar$ is
\begin{equation}\label{dbarstarformula} 
\dbar^*_{2\lambda\phi}(\sum_j f_j \bar z_j) 
= -\sum_j \big(e^{2\lambda\phi}\partial_{z_j}e^{-2\lambda \phi} \big)f_j 
= -\sum_j \big(\partial_{z_j}-2\lambda\partial_{x_j}\phi\big)f_j \end{equation}
since $\phi$ depends only on $x$.

For $\xi\in\complex^d$ 
we will primarily with $\complex$--valued functions, and $(0,1)$ forms, 
of the special form $z=x+iy\mapsto e^{i\lambda\xi\cdot y}f(x)$. 
Such a scalar-valued function is holomorphic if and only if
\begin{equation}
0 = \dbar(e^{i\lambda\xi\cdot y}f(x)) = e^{i\lambda\xi\cdot y} \sum_{j=1}^d (\partial_{x_j}-\lambda\xi_j)f\,d\bar z_j.
\end{equation}
Forms and functions not of this special form will appear near the end of the analysis.
The operator $\dbar^*_{2\lambda\phi}$ can be applied to $e^{i\lambda\xi\cdot y}f(x)$
even though such a function rarely lies in $L^2(\complex^d,e^{-2\lambda\phi})$, by using the expression \eqref{dbarstarformula}.

The question around which our analysis revolves is
for which pairs $(\xi,\lambda)$ the function $e^{\lambda\xi\cdot z}$
is close to the range of $\dbar^*_{2\lambda\phi}$, in a suitable sense.
Formulation of this closeness must take into account the infinite $L^2(\complex^d,e^{-2\lambda\phi})$ norm
of the function $e^{\lambda z\cdot\xi}$.

Denote by $\diver$ the divergence operator, which maps $1$--forms with domain $\reals^d$ to scalar-valued functions
with the same domain:
\begin{equation} \diver(\sum_j u_j\,dx_j) = \sum_j \frac{\partial u_j}{\partial x_j} =
\sum_j (\partial_{x_j}u_j)\,dx_j.\end{equation}

Now
\begin{equation} \label{eq:adjointequation}
\dbar^*_{2\lambda\phi} \big(e^{i\lambda\xi\cdot y} u(x)\big)
= e^{i\lambda\xi\cdot y}f(x)\ \text{ if and only if } \ 
-\sum_j (\partial_{x_j}+\lambda\xi_j -2\lambda\partial_{x_j}\phi)u_j= f, \end{equation}
that is, if and only if
\begin{equation} \label{divergenceeqn} -\diver(e^{\lambda\xi\cdot x-2\lambda\phi(x)}u) 
= e^{\lambda\xi\cdot x-2\lambda\phi(x)} f = e^{2\lambda(\xi\cdot x-\phi(x))}.  \end{equation}
We are interested in the possible existence of pairs $(\xi,\lambda)$ for which
equation \eqref{eq:adjointequation} with right-hand side
$f(x) = e^{\lambda\xi\cdot x}$  admits an exact or near solution $u$ which
enjoys suitable upper bounds. 

The range of the divergence operator consists, formally, of all functions satisfying $\int_{\reals^d} g(x)\,dm(x)=0$.
Therefore the discussion turns on the approximate vanishing of
$\int_{\reals^d} e^{2\lambda(\xi\cdot x-\phi(x))}\,dm(x)$. 

This integral has an alternative interpretation as the analytic continuation to the complex domain, 
with respect to $\xi$, of the function
$\reals^d\owns\xi\mapsto \norm{e^{\lambda x\cdot\xi}}_{L^2(\reals^d,e^{-2\lambda\phi})}^2$.

\section{Preparations}

\subsection{Solvability of the divergence equation with $L^2$ bounds}

Define
\begin{equation} \Phi(x) =\Phi_{\xi,\lambda}(x)= \lambda(\Re\xi\cdot x-\phi(x)). \end{equation}
$\Phi$ depends only on the real part of $\xi$, and is real-valued.
The Hessian matrix of $\Phi$ is comparable to $-\lambda$ times the identity matrix, uniformly in $\lambda,\xi,x$.
Consequently there exist a unique point $x^\dagger\in\reals^d$, depending on $\xi,\lambda$, satisfying
\begin{equation} \Phi(x^\dagger)=\max_{x\in\reals^d} \Phi(x), \end{equation}
and constants $c_1,c_2,C\in\reals^+$ such that
\begin{equation} C^{-1} e^{-c_1\lambda|x-x^\dagger|^2}\le e^{2\Phi(x)}\le Ce^{-c_2\lambda|x-x^\dagger|^2} \end{equation}
uniformly for all $\xi,\lambda,x$.

Let $a>0$ be an exponent, depending only on the dimension $d$, which is to be chosen below. All that will matter
concerning $a$ is that it if $a$ is chosen to be sufficiently large then certain properties hold.

Define the auxiliary weight
\begin{equation} \label{eq:gammadefn} \gamma(x) = a\ln(1+|x-x^\dagger|^2). \end{equation}
In particular, we require that $a > d/2$, which ensures that 
\[\int_{\reals^d} e^{-\gamma(x)}\,dm = \int_{\reals^d} (1+|x-x^\dagger|^2)^{-a}\,dm(x)<\infty.\]
Therefore the function $e^{2\lambda(\xi\cdot x-\phi(x))}$ satisfies
\begin{equation}
\int_{\reals^d} |e^{2\lambda(\xi\cdot x-\phi(x))}|^2 e^{-4\Phi(x)-\gamma(x)}\,dm(x)
= \int_{\reals^d} e^{-\gamma(x)}\,dm(x)<\infty,
\end{equation}
and this quantity is independent of $\xi,\lambda$ even though $\gamma$ depends through $x^\dagger$ on the real part of $\xi$.

We will solve the equation \begin{equation}-\diver(u) = e^{2\lambda(\xi\cdot x-\phi(x))}\end{equation} with $u$
in the space of one-forms satisfying $\int_{\reals^d} |u(x)|^2 e^{-4\Phi(x)-2\gamma(x)}\,dm(x)<\infty$.
This is the reverse of the usual situation; the weight $\Phi$ is concave rather than convex,
so the standard weighted theory \cite{hormander}, adapted from the complex case to the real case, does not apply.

Let $\scripth_1$ be the Hilbert space of all equivalence classes of Lebesgue measurable complex--valued 
$(0,1)$ forms defined on $\reals^d$, with norm
\begin{equation} \norm{u}_{\scripth_1}^2 = \int_{\reals^d} |u(x)|^2 e^{-4\Phi(x)-2\gamma(x)}\,dm(x).  \end{equation}
Let $\scripth_2$ be the Hilbert space of all equivalence classes of Lebesgue measurable functions
$f:\reals^d\to\complex$, with norm
\begin{equation} \norm{f}_{\scripth_2}^2 = \int_{\reals^d} |f(x)|^2 e^{-4\Phi(x)-\gamma(x)}\,dm(x).  \end{equation}

If $a$ is sufficiently large then $f(x) = e^{2\lambda (x\cdot\xi-\phi(x))}$ satisfies
\begin{equation*}
\norm{e^{\lambda x\cdot\xi}}_{\scripth_2}^2 = 
\int_{\reals^d} e^{4\lambda(x\cdot\Re\xi-\phi(x))} e^{-4\lambda(x\cdot\Re\xi - \phi(x))}e^{-\gamma(x)}\,dm(x)
= \int_{\reals^d} e^{-\gamma(x)}\,dm(x)<\infty;
\end{equation*}
this norm is independent of $\xi,\lambda$. In particular, $e^{2\lambda(x\cdot\xi-\phi(x))}\in\scripth_2$.

Regard $\diver$ as an unbounded linear operator from $\scripth_1$ to $\scripth_2$,
whose domain is the closure of the space of continuously differentiable compactly supported one-forms with respect to the graph norm.
The formal adjoint $\diver^*$ of $\diver$ in this Hilbert space setting is
\begin{equation}\label{eq:diverstarformula}
\diver^*(f) = e^{\gamma(x)}\sum_{j=1}^d (-\partial_{x_j} + 4\partial_{x_j}\Phi+\partial_{x_j}\gamma)f\,dx_j.
\end{equation}

$\scripth_2\subset L^1(\reals^d)$ by virtue of the Cauchy-Schwarz inequality
and the rapid decay of $e^{4\Phi}$,  and thus $\int_{\reals^d} f\,dm$ is well-defined for all $f\in\scripth_2$.
Define $\frakF\subset\scripth_2$ to be the set of all $f\in\scripth_2$ that satisfy
\begin{equation} \int_{\reals^d} f(x)\,=0.  \end{equation}
$\frakF$ is a closed subspace of $\scripth_2$, of codimension one, which contains the 
image under $\diver$ of the set of all compactly supported continuously differentiable forms,
and hence by closedness contains the range of $\diver$.


\begin{lemma} \label{lemma:L2solution}
Let $d\ge 1$. 
Let $\phi$ satisfy the hypotheses 
\eqref{phiH1},\eqref{phiH2},\eqref{phiH3},\eqref{phiH4}.
There exist constants $a,C<\infty$ with the following properties.
Let $\lambda$ be sufficiently large, let $\xi\in\complex$, and suppose that 
$f\in\scripth_2$ satisfies $\int_{\reals^d} f\,dm=0$.
There exists a $1$-form $u\in\scripth_1$ satisfying
\begin{align}
& \diver(u) = f
\ \text{ on $\reals^d$},
\\& \int_{\reals^d} |u(x)|^2 e^{-4\Phi(x)-2\gamma(x)}\,dm(x) 
\le C \int_{\reals^d} |f(x)|^2 e^{-4\Phi(x)-\gamma(x)}\,dm(x).
\end{align}
\end{lemma}

Recall that $a$ is the parameter that appears in the definition \eqref{eq:gammadefn} of $\gamma$.
It will be essential for the ensuing argument that $a,C$ may be chosen to be independnet
of $\lambda,\xi$.

Only the real part of $\xi$ enters into the formulation of Lemma~\ref{lemma:L2solution}, so throughout its
proof we will assume that $\xi\in\reals^d$.  The main step in that proof will be the following lemma,
whose justification is deferred until \S\ref{section:finalproof}.

\begin{lemma} \label{lemma:weightedlower}
Let $\phi$ satisfy the hypotheses \eqref{phiH1},\eqref{phiH2},\eqref{phiH3},\eqref{phiH4}.
If the exponent $a$ is chosen to be sufficiently large then
there exists $C<\infty$ such that for all sufficiently large $\lambda$ and all $\xi\in\reals^d$,
for any function $f$ in the intersection of $\frakF$ with the domain of $\diver^*$,
\begin{equation}\label{ineq:weightedlower} \norm{f}_{\scripth_2}\le C\norm{\diver^*f}_{\scripth_1}.\end{equation}
\end{lemma}

According to Lemmas~4.1.1 and 4.1.2 of \cite{hormander},
it follows that the range of $\diver$, as a closed unbounded linear operator
from $\scripth_1$ to $\scripth_2$, equals $\frakF$,  and moreover that for any $f\in \frakF$ there exists $u\in \scripth_1$
satisfying $\diver(u)=f$ with $\norm{u}_{\scripth_1}\le C\norm{f}_{\scripth_2}$, with $C$ independent
of $\xi\in\reals^d$ and $\lambda\in\reals^+$ provided that $\lambda$ is sufficiently large.
Lemma~\ref{lemma:L2solution} is thus a corollary of Lemma~\ref{lemma:weightedlower}.

A solution of the divergence equation with additional desirable properties
can be obtained, and will be needed in the application below.
Because the divergence equation is undetermined, one cannot hope that arbitrary
solutions will have favorable properties; it is necessary to select an appropriate solution.
Consider the operator $T=\diver\circ\diver^*$, 
which satisfies $\norm{f}_{\scripth_1} \le C\norm{Tf}_{\scripth_1}$
since for $f$ in the intersection of $\frakF$ with the domain of $T$, 
\[\norm{Tf}_{\scripth_2}\norm{f}_{\scripth_2}
\ge \langle Tf,f\rangle_{\scripth_2} = \norm{\diver^* f}_{\scripth_1}^2 \ge C\norm{f}_{\scripth_2}^2.\]
$T$ is self-adjoint, and because of this inequality, maps the intersection of $\frakF$ with its domain
onto $\frakF$. 


\begin{lemma} \label{lemma:betterdiversoln}
Let $\phi$ satisfy the hypotheses \eqref{phiH1},\eqref{phiH2},\eqref{phiH3},\eqref{phiH4}.
Let the parameter $a$ be sufficiently large. Then for all sufficiently large $\lambda\in\reals+$,
all $\xi\in\complex^d$, and any $\phi,f$ satisfying the hypotheses of Lemma~\ref{lemma:L2solution}, 
there exists a solution of $\diver(u)=f$ satisfying 
\begin{equation} \int_{\reals^d} \big(|u(x)|^2 + |\nabla u(x)|^2\big) e^{-4\Phi(x)-3\gamma(x)}\,dm(x) 
\le C \lambda^C \int_{\reals^d} |f(x)|^2 e^{-4\Phi(x)-\gamma(x)}\,dm(x).  \end{equation}
\end{lemma}

\begin{proof}
Solve $\diver\diver^*(h)=f$, and set $u=\diver^*(h)$. This is a second order elliptic equation for $h$,
whose coefficients are $O(\lambda^2+|x|^2)$, together with all of their derivatives.
Cover $\reals^d$ by a union of suitable balls, so that if $x$ belongs to a ball $B$ in this cover
then the radius of $B$ is comparable to $\lambda^{-1} (1+|x-x^\dagger|)^{-1}$. 
This ensures that
\begin{equation*} \max_B e^{-4 \Phi} \le C\min_B e^{-4\Phi} \end{equation*}
for a finite constant $C$ independent of $\lambda,\xi,B$.
Apply standard elliptic regularity estimates in each ball to obtain upper bounds for the second partial
derivatives of $h$, and sum the results.
The extra factor $e^{-\gamma(x)}$ on the left-hand side of the inequality, together with the factor $\lambda^C$
on the right-hand side, compensate for the resulting losses due to growth in the coefficients as $\lambda\cdot(1+|x|)\to\infty$.
\end{proof}


\section{Absence of near-resonances} 

\subsection{Exponential decay implies absence of near-resonances}
For $\xi\in\complex^d$ and $\lambda\in\reals^+$ define 
\begin{equation} \scriptf(\xi,\lambda) = \int_{\reals^d} e^{2\lambda(\xi\cdot x-\phi(x))}\,dm(x).\end{equation}

The goal of this section is to establish the following lemma, which links the off-diagonal
behavior of  Bergman kernels with the function $\scriptf$. In a subsequent step we will establish
a link between $\scriptf$ and the metric $\phi$.
\begin{proposition} \label{prop:scriptfbounds}
Suppose that there exist sequences $\lambda_\nu,A_\nu\in\reals^+$ and  $\xi_\nu\in\complex^d$  such that
\begin{align}
& \lambda_\nu\to\infty,
\\& A_\nu\to +\infty,
\\&\Re(\xi_\nu)\to \xi^*\in\reals^d,
\\&\Im(\xi_\nu)\to 0
\\&|\scriptf(\xi_\nu,\lambda_\nu)|\le e^{-A_\nu\lambda_\nu}.
\end{align}
Define $x^*\in\reals^d$ by 
\begin{equation} \nabla \phi(x^*) = \xi^*.\end{equation}
Then there is no neighborhood of $x^*$ in $\complex^d$ in which $(B_{\lambda_\nu}: \nu\in\naturals)$
decays exponentially fast away from the diagonal.
\end{proposition}

\subsection{Beginning of the proof of Proposition~\ref{prop:scriptfbounds}}

\begin{proof}
Suppose to the contrary that $(\lambda_\nu)$, $(A_\nu)$, $(\xi_\nu)$ are sequences with the stated properties.
Writing $z=x+iy\in\reals^d+i\reals^d$, define
\begin{align}
\psi_\nu(z) &= e^{\lambda_\nu z\cdot\xi_\nu},
\\
f_\nu(x) &= e^{2\lambda_\nu (x\cdot\xi_\nu-\phi(x))}
\end{align}
where $z\cdot\xi = \sum_{j=1}^d z_j\xi_j$.
Let $\gamma_\nu(x) = a\ln(1+|x-x_\nu|^2)$ in \eqref{eq:gammadefn}, 
where $x_\nu$ is the unique solution of
\begin{equation}\nabla\phi(x_\nu)=\Re\xi_\nu.\end{equation}
Since $\Re\xi_\nu\to\xi^*\in\reals^d$, $x_\nu\to x^*$.

Set \begin{equation} \Phi_\nu(x) = \lambda_\nu(x\cdot\Re\xi_\nu-\phi(x)). \end{equation}
Set \[b_\nu = (\int_{\reals^d} e^{2\Phi_\nu})^{-1}\int_{\reals^d} f_\nu 
= (\int_{\reals^d} e^{2\Phi_\nu})^{-1}\scriptf(\xi_\nu,\lambda_\nu).\]
Then $\int_{\reals^d} (f_\nu-b_\nu e^{2\Phi_\nu})\,dm=0$, and $f_\nu-b_\nu e^{2\Phi_\nu}\in \scripth_2$,
so this function belongs to the range of the divergence operator. 
Let $u_\nu: \reals^d\to\complex$ be the solution of the equation 
\begin{equation} \diver(u_\nu) = f_\nu - b_\nu e^{2\Phi_\nu} \end{equation}
guaranteed by Lemma~\ref{lemma:betterdiversoln}. 
The Lemma gives 
\begin{align*}
\int_{\reals^d} \big(|u_\nu(x)|^2  &+ |\nabla u_\nu(x)|^2\big)  e^{-4\Phi_\nu(x)-3\gamma_\nu(x)}\,dm(x) 
\\  &\le C \int_{\reals^d} |f_\nu(x)|^2 e^{-4\Phi_\nu(x)-\gamma_\nu(x)}\,dm(x) 
+ C\lambda^d|\int_{\reals^d}f_\nu\,dm|^2
\\ & \le C + Ce^{-2A_\nu\lambda_\nu} \lambda^d
\\& \le C 
\end{align*}
where $C<\infty$ is independent of $\nu$.

Defining 
\begin{equation}
v_\nu(z) = e^{i\lambda_\nu y\cdot\xi_\nu} e^{-\lambda_\nu (x\cdot\xi_\nu-2\phi(x))}u_\nu(z),
\end{equation}
these conclusions become, uniformly for $y\in\reals^d$:
\begin{gather}
\dbar^*_{2\lambda_\nu\phi} v_\nu(x+iy) = \psi_\nu(z)  - b_\nu e^{i\lambda_\nu y\cdot\xi_\nu} e^{\lambda_\nu x\cdot(2\Re\xi_\nu-\xi_\nu)}
\\
\int_{\reals^d} \big(|v_\nu(x+iy)|^2 + |\nabla v_\nu(x+iy)|^2\big) 
e^{-2\lambda_\nu x\cdot\Re\xi_\nu} e^{-2\gamma_\nu(x)}\,dm(x) 
\le C e^{2\lambda_\nu|\Im(\xi_\nu)|\cdot|y|}.
\end{gather}

Let $V\subset\complex^d$ be a ball centered at $x^*$, independent of $\nu$, to be chosen below. 
Then
\begin{equation} \label{psinunorm}
\norm{\psi_\nu}_{L^2(V,e^{-2\lambda\phi})}^2 \asymp \lambda_\nu^{-d}e^{2\lambda_\nu(\xi_\nu\cdot\Re\xi_\nu-\phi(x_\nu))}
\end{equation}
for all sufficiently large $\nu$, since $x_\nu\to x^* \in V$.
The preceding can be rewritten in the more illuminating form
\begin{multline}
\int_{\reals^d} |v_\nu(x+iy)|^2 e^{-2\lambda_\nu ((x-x_\nu)\cdot\Re\xi_\nu-(\phi(x)-\phi(x_\nu))-3\gamma_\nu(x)} 
e^{-2\lambda_\nu \phi(x)} \,dm(x) 
\\
\le C e^{2\lambda_\nu|\Im(\xi_\nu)|\cdot|y|}  e^{2\lambda_\nu (x_\nu\cdot\Re\xi_\nu-\phi(x_\nu))}
\\
\asymp  \lambda_\nu^{d} e^{2\lambda_\nu|\Im(\xi_\nu)|\cdot|y|}  \norm{\psi_\nu}_{L^2(V,e^{-2\lambda_\nu\phi})}^2.
\end{multline}
The left-hand side is the norm squared $\int_{\reals^d} |v_\nu(x+iy)|^2 e^{-2\lambda_\nu\phi(x)}\,dm(x)$,
with an advantageous supplementary weight in the integral of the form
\[e^{-2\lambda_\nu ((x-x_\nu)\cdot\Re\xi_\nu-(\phi(x)-\phi(x_\nu))-3\gamma_\nu(x)} \ge e^{c\lambda_\nu |x-x_\nu|^2}.\]

However, this weight is of no help in overcoming the
disadvantageous factor $\lambda_\nu^d e^{2\lambda_\nu|\Im(\xi_\nu)|\cdot|y|}$ on the right-hand side
when $x\approx x_\nu$ but $y\ne 0$;
overcoming that factor is the crucial issue. Therefore this supplementary factor will be
of no further use, so will be dropped.  Accordingly, one has
\begin{multline}
\int_{\reals^d} \big(|v_\nu(x+iy)|^2 
 + |\nabla v_\nu(x+iy)|^2 \big)
e^{-2\lambda_\nu \phi(x)} \,dm(x) 
\\ \le C\lambda_\nu^{d} e^{2\lambda_\nu|\Im(\xi_\nu)|\cdot|y|}  \norm{\psi_\nu}_{L^2(V,e^{-2\lambda_\nu\phi})}^2.
\end{multline}

\subsection{Localized solutions $F_\nu$}
Without loss of generality we may assume by change of coordinates that $x_\nu\to x^*=0$. 
Let $\eta\in C^\infty_0(\complex^d)$ be a function which is identically equal to $1$
in a neighborhood of $0$, and is supported within an open neighborhood of $0$ in which the Bergman kernels 
$B_{\lambda_\nu}$ decay exponentially fast away from the diagonal as $\nu$ (and hence $\lambda_\nu$) tends
to infinity. Let $V$ be a bounded open set containing the support of $\eta$.  
Consider the functions $F_\nu:\complex^d\to\complex$ defined by
\begin{equation} F_\nu = \dbar^*_{2\lambda_\nu\phi}(\eta v_\nu).\end{equation} 
These are supported in $V$, which is a fixed bounded set independent of $\nu$.

Let $W\Subset W'$  and $V'$ be bounded open subsets of $\complex^d$
such that $0\in W$, 
the closure of $W$ is contained in $W'$, 
$\eta\equiv 1$ in a neighborhood of the closure of $W'$,
the closure of $V'$ is disjoint from the closure of $W'$,
and the support of $\nabla\eta$ is contained in $V'$.
For all sufficiently large indices $\nu$,
\begin{multline}
\norm{F_\nu-\eta \psi_\nu}_{L^2(\complex^d, e^{-2\lambda_\nu\phi})} 
+ \norm{\nabla F_\nu-\nabla(\eta \psi_\nu)}_{L^2(\complex^d, e^{-2\lambda_\nu\phi})} 
\\ \le \lambda_\nu^C e^{C\lambda_\nu|\Im(\xi_\nu)|}\norm{\psi_\nu}_{L^2(V,e^{-2\lambda_\nu\phi})}, 
\end{multline}
and
\begin{multline}
\norm{F_\nu-\eta \psi_\nu}_{L^2(W', e^{-2\lambda_\nu\phi})} 
+ \norm{\nabla F_\nu-\nabla(\eta \psi_\nu)}_{L^2(W', e^{-2\lambda_\nu\phi})} 
\\ \le \lambda_\nu^C e^{\lambda_\nu(-cA_\nu+C|\Im(\xi_\nu)|}\norm{\psi_\nu}_{L^2(V,e^{-2\lambda_\nu\phi})}.
\end{multline}
In particular, since $\dbar\psi_\nu=0$,
\begin{align}
&\norm{\dbar F_\nu}_{L^2(\complex^d, e^{-2\lambda_\nu\phi})} 
\le \lambda_\nu^C e^{C\lambda_\nu|\Im(\xi_\nu)|}\norm{\psi_\nu}
\label{Fnuglobalbound}
\\
&\norm{\dbar F_\nu}_{L^2(W', e^{-2\lambda_\nu\phi})} 
\le \lambda_\nu^C e^{\lambda_\nu(-cA_\nu+C|\Im(\xi_\nu)|)}\norm{\psi_\nu}.
\label{FnuWprimebound}
\end{align}


\subsection{Solution of a final $\dbar$ equation}
The hypothesis that the Bergman kernels decay exponentially away from the diagonal will be applied
not to the sequence $\psi_\nu$, but to a certain related sequence of functions $G_\nu$. These will be constructed
by solving a final $\dbar$ equation. These are not of the product form $e^{i\lambda_\nu y\cdot \xi_\nu}f_\nu(x)$.
To prepare for their construction, choose a $C^\infty$ function $\tilde\phi:\complex^d\to\reals$ 
and a constant $\eps>0$ with the following properties:
\begin{enumerate}
\item $\tilde\phi$ is plurisubharmonic.
\item $\tilde\phi\le \phi$.
\item $\tilde\phi\equiv\phi$ in a neighborhood of the support of $\nabla\eta$.
\item There exists $\eps>0$ such that $\tilde\phi(z) \le \phi(z) - \eps$ for all $z\in \complex^d\setminus V'$.
\end{enumerate}
These exist, because $\phi$ is strictly plurisubharmonic.

The right-hand side in our $\dbar$ equation will be $\dbar F_\nu$.
The norm of $\dbar F_\nu$ is still under satisfactory control with respect to the modified weight $\tilde\phi$:
\begin{equation} \label{eq:Ftildephibound}
\norm{\dbar F_\nu}_{L^2(\complex^d,e^{-2\lambda_\nu\tilde\phi})}
\le e^{C\lambda_\nu|\Im(\xi_\nu)|} \norm{\psi_\nu}_{L^2(V,e^{-2\lambda_\nu\phi})}
\end{equation}
for all sufficiently large $\nu$. Note that the norm on the left-hand side is with respect
to $\tilde\phi$, while only $\phi$ appears on the right-hand side.
This relies on \eqref{Fnuglobalbound} and \eqref{FnuWprimebound}, together with the crucial assumption\footnote{ 
What is actually needed is not that $A_\nu\to\infty$,
but rather that $\limsup_{\nu\to\infty}A_\nu$ is larger than a certain constant which is
determined by $\phi$, the neighborhood in which the Bergman kernels are assumed to decay exponentially,
and their rate of exponential decay.}
that $A_\nu\to\infty$ as $\nu\to\infty$. 

\begin{lemma} \label{lemma:weightedtilde}
Let $\tilde\phi$ have the properties listed.  For each sufficiently large $\nu$ there exists a solution $G_\nu$ of the equation 
\begin{equation}\label{eq:Gnudefn} \dbar G_\nu = \dbar F_\nu\end{equation} 
satisfying
\begin{equation} \int_{\complex^d} |G_\nu|^2 e^{-2\lambda_\nu \tilde\phi}e^{-\gamma}\,dm
\le C \int_{\complex^d} |\dbar F_\nu|^2 e^{-2\lambda_\nu \tilde\phi}\,dm.  \end{equation}
\end{lemma}

\begin{proof}
A direct application of the well-known weighted theory for the $\dbar$ equation \cite{hormander} suffices.
\end{proof}

For each sufficiently large $\nu$, choose a solution $G_\nu$ of \eqref{eq:Gnudefn}.
Concerning these functions, two consequences of Lemma~\ref{lemma:weightedtilde} together with \eqref{eq:Ftildephibound}
will be useful.  Firstly, in the whole space
\begin{equation} \label{eq:Gglobalbound}
\int_{\complex^d} |G_\nu|^2 e^{-2\lambda_\nu\phi}\,dm
\le C\lambda_\nu^C e^{C\lambda_\nu|\Im(\xi_\nu)|}\norm{\psi_\nu}^2_{L^2(V,e^{-2\lambda_\nu\phi})}.
\end{equation}
Secondly, in the complement of $V'$ there is the improved upper bound 
\begin{equation} \label{eq:GWbound}
\int_{\complex^d\setminus V'} |G_\nu|^2 e^{-2\lambda_\nu\phi}\,dm
\le e^{-c\lambda_\nu} e^{C\lambda_\nu|\Im(\xi_\nu)|}\norm{\psi_\nu}^2_{L^2(V,e^{-2\lambda_\nu\phi})}.
\end{equation}

\subsection{Arrival at a contradiction}
We now complete the proof of Proposition~\ref{prop:scriptfbounds}, modulo the deferred proof of Lemma~\ref{lemma:weightedlower}, 
by showing how its hypotheses, together with the assumption that the Bergman kernels $B_{\lambda_\nu}$ decay exponentially
fast away from the diagonal, lead to a contradiction.
Throughout the discussion, it is assumed that $\nu$ is sufficiently large.
An upper bound of the form ``$O(M)$ in $W$'' 
indicates a function whose norm in $L^2(W,e^{-2\lambda_\nu\phi})$ is $O(M)$, uniformly in $\nu$.

Recalling that
$F_\nu= \dbar^*_{2\lambda_\nu\phi}(\eta v_\nu)$ and that $\dbar G_\nu=\dbar F_\nu$, the equation 
\begin{equation} 
G_\nu = (G_\nu-\dbar^*_{2\lambda_\nu\phi}(\eta v_\nu)) \ +\ \dbar^*_{2\lambda_\nu\phi}(\eta v_\nu) \end{equation}
expresses $G_\nu$ as the sum of an element of the nullspace of $\dbar$ plus a function orthogonal to that nullspace.
Therefore
\begin{equation} (I-B_{\lambda_\nu})G_\nu = \dbar^*_{2\lambda_\nu\phi}(\eta v_\nu).  \end{equation}
Consequently
\begin{equation} (I-B_{\lambda_\nu})G_\nu = \dbar^*_{2\lambda_\nu\phi}v_\nu  \ \text{in $W$} \end{equation}
since $\eta\equiv 1$ in $W$.

Since $\dbar^*_{2\lambda_\nu\phi}v_\nu =\psi_\nu
- b_\nu e^{i\lambda_\nu y\cdot\xi_\nu} e^{\lambda_\nu x\cdot(2\Re\xi_\nu-\xi_\nu)}$
and $|b_\nu|\le e^{-cA_\nu\lambda_\nu}$,
\begin{equation}
(I-B_{\lambda_\nu})G_\nu =  \psi_\nu\ +O(e^{-c\lambda_\nu A_\nu}\norm{\psi_\nu}_{L^2(V,e^{-2\lambda_\nu\phi})})\  \text{ in $W$.}
\end{equation}
Using the strong bound provided by
inequality \eqref{eq:GWbound}  in the complement of $V'$, and in particular in $W$, this can be rewritten as
\begin{equation} \label{eq:psinuboundatlast}
\psi_\nu = 
-B_{\lambda_\nu}G_\nu + O\big(e^{-c\lambda_\nu}e^{-C\lambda_\nu|\Im(\xi_\nu)|} 
\norm{\psi_\nu}_{L^2(V,e^{-2\lambda_\nu\phi})}\big)\ \text{in $W$.} 
\end{equation}

Let $\one_{V'}$ denote the indicator function of $V'$.
Because $B_{\lambda_\nu}$ is a contraction on $L^2(\complex^d,e^{-2\lambda_\nu\phi})$,
\begin{align*}
\norm{B_{\lambda_\nu}G_\nu}_{L^2(W,e^{-2\lambda_\nu\phi})}
&\le
\norm{B_{\lambda_\nu}(\one_{V'} G_\nu)}_{L^2(W,e^{-2\lambda_\nu\phi})}
+ \norm{B_{\lambda_\nu} (\one_{\complex\setminus V'} G_\nu)}_{L^2(\complex^d,e^{-2\lambda_\nu\phi})}
\\& \le
\norm{B_{\lambda_\nu}(\one_{V'} G_\nu)}_{L^2(W,e^{-2\lambda_\nu\phi})}
+ \norm{G_\nu}_{L^2(\complex^d\setminus V',e^{-2\lambda_\nu\phi})}
\\& \le
\norm{B_{\lambda_\nu}(\one_{V'} G_\nu)}_{L^2(W,e^{-2\lambda_\nu\phi})}
+ 
e^{-c\lambda_\nu} e^{C\lambda_\nu|\Im(\xi_\nu)|}\norm{\psi_\nu}_{L^2(V,e^{-2\lambda_\nu\phi})}.
\end{align*}
The sets $W,V'$ were constructed to have disjoint closures, and so that both are contained
in a region in which the Bergman kernels $B_{\lambda_\nu}$
decay exponentially fast away from the diagonal. Therefore for all sufficiently large $\nu$,
\begin{align*}
\norm{B_{\lambda_\nu}(\one_{V'} G_\nu)}_{L^2(W,e^{-2\lambda_\nu\phi})}
&\le e^{-c\lambda_\nu} \norm{G_\nu}_{L^2(V',e^{-2\lambda_\nu\phi})}
\\&\le e^{-c\lambda_\nu} e^{C\lambda_\nu|\Im(\xi_\nu)|}\norm{\psi_\nu}_{L^2(V,e^{-2\lambda_\nu\phi})}.
\end{align*}
Inserting these bounds into \eqref{eq:psinuboundatlast} gives
\begin{equation}
\norm{\psi_\nu}_{L^2(W,e^{-2\lambda_\nu\phi})}  
\le e^{-c\lambda_\nu} e^{C\lambda_\nu|\Im(\xi_\nu)|}\norm{\psi_\nu}_{L^2(V,e^{-2\lambda_\nu\phi})}
\end{equation}
where $c>0$.

We have normalized $\phi$ so that $\phi(0)=0$ and $\nabla\phi(0)=0$, so $\phi(x)\asymp |x|^2$.
It is thus apparent from the explicit formula $\psi_\nu(z) = e^{\lambda_\nu z\cdot\xi_\nu}$
and the assumption that $x_\nu\to 0$ that the functions $\psi_\nu$ peak near $0$ in the sense that
\begin{equation}
\norm{\psi_\nu}_{L^2(W,e^{-2\lambda_\nu\phi})} 
\ge e^{-C\lambda_\nu|\Im(\xi_\nu)|} \norm{\psi_\nu}_{L^2(V,e^{-2\lambda_\nu\phi})}. 
\end{equation}
Therefore we have shown that
\begin{equation}
\norm{\psi_\nu}_{L^2(V,e^{-2\lambda_\nu\phi})}  
\le e^{-c\lambda_\nu} e^{C\lambda_\nu|\Im(\xi_\nu)|}\norm{\psi_\nu}_{L^2(V,e^{-2\lambda_\nu\phi})}
\end{equation}
where $c>0$.
Since $|\Im(\xi_\nu)|\to 0$ as $\nu\to\infty$, this is a contradiction for all sufficiently large $\nu$.
\end{proof}

This completes the proof of Proposition~\ref{prop:scriptfbounds}, modulo the deferred proof of Lemma~\ref{lemma:weightedlower} 
concerning solvability of the divergence equation.

\section{Conclusion of the proof}
The second of the two main steps of the proof is to link properties of $\scriptf$ with analyticity of the metric $\phi$.
The strong convexity of $\phi$ implies that the mapping $x\mapsto \nabla\phi(x)$ from $\reals^d$ to $\reals^d$
is a bijection.
Define
the function $\tau:\reals^d\to\reals^d$  to be the inverse of $\nabla\phi$, that is, 
\begin{equation} \label{taudefn} \nabla\phi(\tau(\xi)) = \xi.\end{equation}

This subsection is devoted to the proof of the following result.
\begin{proposition} \label{prop:step2}
Let $(\lambda_\nu: \nu\in\naturals)$ be a sequence of positive real numbers tending to infinity.
Suppose that $(B_{\lambda_\nu}: \nu\in\naturals)$ 
decays exponentially fast away from the diagonal, in some neighborhood of $a\in\reals^d$.  
Then the function $\tau$ is real analytic in some neighborhood of $\xi=\nabla \phi(a)$.
\end{proposition}
Consequently the inverse function $\nabla\phi$, and hence $\phi$ itself, are real analytic in a corresponding 
neighborhood of $a$.

By the hypothesis of exponentially fast decay we mean that there exists $\varrho>0$
such that for each $\delta>0$ there exists $c<\infty$ such that 
\begin{equation} |B_{\lambda_\nu}(z,z')|\le e^{-c\lambda_\nu} \end{equation}
for all ordered pairs of elements of $\complex^d$ satisfying
$\rho(a,z)<\varrho$, $\rho(a,z')<\varrho$, and $\delta\le\rho(z,z')$.

Let $a\in\reals^d$. 
By making the change of variables $z\mapsto z-a$ and subtracting from $\phi$ a real-valued affine function,
we may assume without loss of generality that $a=0$ and that $\phi(0)=\nabla \phi(0)=0$.

\begin{lemma} \label{lemma:pluriharmoniclimit}
Under the hypotheses of Proposition~\ref{prop:step2} with $a=0$ and $\nabla\phi(0)=0$,
there exist an open ball $\scriptb\subset\complex^d$ centered at $0$,
a sequence of indices $\nu_k$ tending to $\infty$, and a real analytic function $u:\scriptb\to\reals$ such that
\begin{equation} \tfrac12 \lambda_{\nu_k}^{-1}\log|\scriptf(\xi,\lambda_{\nu_k})|\to u(\xi) \end{equation}
uniformly in $\scriptb$ as $k\to\infty$.  \end{lemma}

Let $\scriptb(\varrho)$ denote the open ball in $\complex^d$ of radius $\varrho$, centered at the origin.
The functions $F_\nu(\xi)=\scriptf(\xi,\lambda_\nu)$ are all holomorphic in some common neighborhood of $\xi=0$,
independent of $\nu$. Moreover, straightforward estimation demonstrates that there exists $C<\infty$ such that 
\begin{equation}\label{trivialexpupperbound} |F_\nu(\xi)| \le e^{C\lambda_\nu}\end{equation}
for all $\xi$ in that neighborhood and for all $\nu$.

\begin{proof}[Proof of Lemma~\ref{lemma:pluriharmoniclimit}]
According to Proposition~\ref{prop:scriptfbounds}, there exists a ball $\scriptb$ centered at $0$ such that
the function $\xi\mapsto \scriptf(\xi,\lambda_\nu)$ has no zeros in $\scriptb(\varrho)$, and moreover
there exists $C<\infty$ such that
\begin{equation} \lambda_\nu^{-1} \log|F_\nu(\xi)| \ge -C\ \text{ for all $\xi\in\scriptb$}. \end{equation}
Combining this with the upper bound \eqref{trivialexpupperbound} gives
\begin{equation} \big|\,\lambda_\nu^{-1}\log|F_\nu(\xi)|\,\big|\le C \end{equation}
uniformly for all $\xi\in\scriptb$, for all sufficiently large indices $\nu$.

Since $F_\nu$ is holomorphic and zero-free in $\scriptb(\varrho)$,  $u_\nu = \tfrac12 \lambda_\nu^{-1}\log|F_\nu|$ is pluriharmonic there.
Because these functions are uniformly bounded, they form a normal family. 
Therefore after replacing $\scriptb$ by a concentric ball of strictly smaller radius,  
there exist a pluriharmonic function $u$ in $\scriptb(\varrho)$ and a sequence $\nu_k\to\infty$
such that $u_{\nu_k}\to u$
uniformly on all compact subsets of $\scriptb(\varrho)$.

Pluriharmonicity implies real analyticity.
\end{proof}

\begin{lemma} \label{lemma:identifylimit}
The function $u$ in the conclusion of Lemma~\ref{lemma:pluriharmoniclimit} is 
\begin{equation}\label{eq:identifylimit} u(\xi) = \xi\cdot\tau(\xi)-\phi\circ\tau(\xi).  \end{equation} \end{lemma}

\begin{proof}
Consider any $\xi\in\reals^d$.
For large $\lambda$, $\scriptf(\xi,\lambda) = \int_{\reals^d} e^{2\lambda(\xi\cdot t-\phi(t))}\,dt$ 
can be calculated via the method of real stationary phase: Set $\tau=\tau(\xi)$.  As $\lambda\to+\infty$,
\begin{multline} \int_{\reals^d} e^{2\lambda(\xi\cdot t-\phi(t))}\,dt
= c_de^{2\lambda (\xi\cdot\tau-\phi(\tau))}\lambda^{-d/2} \big(\det\nabla^2\phi(\tau)\big)^{-1/2} 
\\+ O\big(e^{2\lambda (\xi\cdot\tau-\phi(\tau))}\lambda^{-(d+2)/2}\big).  \end{multline}
Thus 
\begin{equation} \lambda_\nu^{d/2} |\scriptf(\xi,\lambda_\nu)|
=  e^{2\lambda_\nu(\xi\cdot\tau-\phi(\tau))} \big(\alpha(\xi)+O(\lambda_\nu^{-1})\big) \end{equation}
for a certain strictly positive $\alpha(\xi)$.  Taking logarithms of both sides and dividing by $\lambda_\nu$ gives
\begin{equation} u_\nu(\xi) =  \xi\cdot\tau(\xi)-\phi\circ\tau(\xi) +O(\lambda_\nu^{-1}\log\lambda_\nu) \end{equation}
as $\nu\to\infty$, since $\lambda_\nu^{-1}\log\lambda_\nu\to 0$.  
Restricting attention to the subsequence $\nu_k$ obtained above and letting $k\to\infty$ gives \eqref{eq:identifylimit}.
\end{proof}

\begin{lemma} The function $u$ defined in \eqref{eq:identifylimit} satisfies
\begin{equation} \label{simpleueqn} \nabla u\circ\nabla\phi(x) \equiv x.  \end{equation}   \end{lemma}

\begin{proof}
Substitute $\xi=\nabla\phi(x)$ to rewrite the conclusion of the preceding Lemma asconclusion of the preceding Lemma as
\begin{equation} u(\nabla\phi(x)) = x\cdot \nabla\phi(x) - \phi(x).  \end{equation}
Apply $\nabla=\nabla_x$ to both sides to obtain
\begin{equation}
(\nabla u\circ\nabla\phi(x))\star \nabla^2\phi(x)  = \nabla\phi(x) + x\star \nabla^2\phi(x) - \nabla \phi(x) 
= x\star \nabla^2\phi(x)  \end{equation}
where $\star$ denotes the product of a vector with a matrix.
The Hessian $\nabla^2\phi(x)$ is by the positivity hypothesis an invertible matrix for each $x$, so
the conclusion of the lemma follows.
\end{proof}

\begin{lemma}
Let $\tau$ be the inverse of the mapping $\reals^d\owns x\mapsto\nabla\phi(x)$.
If the function $u(\xi) = \xi\tau(\xi)-\phi\circ\tau(\xi)$ is real analytic in a neighborhood of $\xi_0$
then $\phi\in C^\omega$ in a neighborhood of $\tau(\xi_0)$.
\end{lemma}

\begin{proof}
By \eqref{simpleueqn}, $x\mapsto\nabla u(x)$ is a locally invertible function. This function is real analytic 
by hypothesis.  Therefore its inverse, $x\mapsto\nabla\phi$, is also real analytic. 
Analyticity of $\phi$ is an immediate consequence.
\end{proof}

This completes the proof of Proposition~\ref{prop:step2}, and with it the proof of the main theorem,
except for the deferred proof of Lemma~\ref{lemma:weightedlower}. 

\section{Proof of Lemma~\ref{lemma:weightedlower}}\label{section:finalproof}

Let $\Phi,\gamma$ be as defined above.  We seek to prove that
\begin{equation}\label{eq:injectivebound} \int_{\reals^d} |f(x)|^2 e^{-4\Phi(x)-\gamma(x)}\,dm(x)
\le C\int_{\reals^d} |\diver^* f(x)|^2 e^{-4\Phi(x)-2\gamma(x)}\,dm(x),\end{equation}
under the assumptions that $f$ is continuously differentiable, compactly supported, and satisfies 
$\int_{\reals^d} f\,dm=0$.
Substituting $fe^{-2\Phi(x)}=g(x)$, one has $\int_{\reals^d} e^{2\Phi}g\,dm=0$.
Using the expression \eqref{eq:diverstarformula} for $\diver^*$ gives
\begin{align*}
\int_{\reals^d} |\diver^* f(x)|^2 e^{-4\Phi(x)-2\gamma(x)}\,dm(x)
&= \int_{\reals^d} |\diver^* e^{2\Phi}g|^2 e^{-4\Phi-2\gamma}\,dm
\\&= \int_{\reals^d} |(-\nabla+2\nabla\Phi+\nabla\gamma) g|^2\,dm
\\&= \int_{\reals^d} |e^{2\Phi+\gamma}\nabla e^{-2\Phi-\gamma} g|^2\,dm.
\end{align*}

Thus Lemma~\ref{lemma:weightedlower} is equivalent to
\begin{lemma}
There exists $C<\infty$ such that for all sufficiently large $\lambda\in\reals^+$
all $\xi\in\complex^d$, and all continuously differentiable compactly supported functions
$g:\complex^d\to\complex$ satisfying
$\int_{\reals^d} e^{2\Phi}g\,dm=0$,
\begin{equation} \label{lemma:weightedlowerconjugated}
\int_{\reals^d} |g(x)|^2 e^{-\gamma(x)}\,dm(x)  
\le 
C \int_{\reals^d} |e^{2\Phi+\gamma}\nabla e^{-2\Phi-\gamma} g|^2\,dm.
\end{equation}
\end{lemma}

Define
\begin{equation}\label{eq:Phitildedefn} 
\left\{
\begin{aligned} 
&M_\Phi=\max_{x\in\reals^d} \Phi(x) \\
&\tilde\Phi = \Phi-\Phi(x^\dagger )=\Phi-M_\Phi\le 0. \\
&\Phi^* = 2\tilde\Phi + \gamma. \end{aligned} \right. \end{equation} 
Since $\phi$ is uniformly strictly convex, the Hessian matrix of $\Phi$ is uniformly
comparable to $-\lambda$. Therefore
for all sufficiently large $\lambda\in\reals^+$, 
\begin{equation}\label{eq:MPhibound} e^{4M_\Phi} 
\le C\lambda^{d/2}\int_{\reals^d} e^{4\Phi}\,dm
\le C\lambda^{d/2}\int_{\reals^d} e^{4\Phi+\gamma}\,dm.\end{equation}

\begin{lemma}\label{lemma:gSgbound}
There exists $C<\infty$ such that for any sufficiently large $\lambda\in\reals^+$,
any $\xi\in\reals^d$, and any compactly supported continuously differentiable function $g:\reals^d\to\complex$
satisfying $\int_{\reals^d} ge^{2\Phi}\,dm=0$, for any $x\in\reals^d$
\begin{equation}\label{eq:gSgbound} 
|g(x)| \le C\lambda^{d/2} \int_{\reals^d} |u|\,I(x,u)\,|Sg(x+u)|\,du \end{equation}
where
\begin{equation} I(x,u) = \int_1^\infty e^{\Phi^*(x+su)}s^{d-1}\,ds.  \end{equation}
\end{lemma}

\begin{proof}
For any $x,y\in\reals^d$,
\begin{equation}
e^{-2\Phi(x)-\gamma(x)}(x)g(x) =  e^{-2\Phi(y)-\gamma(y)}g(y) + \int_0^1 (x-y)\cdot\nabla(e^{-2\Phi-\gamma}g)(y+t(x-y))\,dt
\end{equation}
and therefore
\begin{multline}
e^{4\Phi(y)+\gamma(y)} e^{-2\Phi(x)-\gamma(x)}g(x) 
\\ =  e^{2\Phi(y)}g(y) + e^{4\Phi(y)+\gamma(y)} \int_0^1 (x-y)\cdot\nabla(e^{-2\Phi-\gamma}g)(y+t(x-y))\,dt
\end{multline}
Integrating over $\reals^d$ with respect to $dm(y)$ and invoking the condition $\int ge^{2\Phi}\,dm=0$ gives
\begin{multline}\label{eq:gintegrated}
e^{-2\Phi(x)-\gamma(x)} g(x) \int_{\reals^d} e^{4\Phi+\gamma}\,dm
\\=  \int_{\reals^d} e^{4\Phi(y)+\gamma(y)} \int_0^1 (x-y)\cdot \nabla(e^{-2\Phi-\gamma}g)(y+t(x-y))\,dt\,dm(y).
\end{multline}

From \eqref{eq:MPhibound} together with \eqref{eq:gintegrated} we conclude the pointwise bound
\begin{equation}\label{eq:gpointwisebound} 
|g(x)| 
\\ \le C\lambda^{d/2} \int_{\reals^d}\int_0^1 
e^{\Theta(x,y,t)}
|x-y|\,|Sg(tx+(1-t)y)|\,dt\,dm(y)
\end{equation}
where
\begin{align*}
\Theta (x,y,t)
& = \Phi^*(x)+2\Phi^*(y)-\Phi^*(tx+(1-t)y)-\gamma(y)
\\& \le \Phi^*(x)+2\Phi^*(y)-\Phi^*(tx+(1-t)y)
\end{align*}
since $\gamma\ge 0$.  Thus \eqref{eq:gpointwisebound} becomes
\begin{equation} |g(x)| 
\le C\lambda^{d/2} \int_{\reals^d}\int_0^1 e^{\Phi^*(x)+2\Phi^*(y)-\Phi^*(tx+(1-t)y)}
|x-y|\,|Sg(tx+(1-t)y)|\,dt\,dm(y).
\end{equation}

Substitute $(t,y)\leftrightarrow (t,u)$ where $tx+(1-t)y = x+u$, so that $|x-y|=(1-t)^{-1}|u|$, and
then substitute $s=(1-t)^{-1}$ to deduce that
\begin{align*} 
|g(x)| &\le  C\lambda^{d/2} \int_{\reals^d}
\Big(\int_0^1 e^{\Phi^*(x)+2\Phi^*(x+(1-t)^{-1}u)-\Phi^*(x+u)}(1-t)^{-d-1}\,dt\Big) |u|\,|Sg(x+u)|\,dm(u)  
\\&= C\lambda^{d/2} \int_{\reals^d}
\Big( \int_1^\infty e^{\Phi^*(x)+2\Phi^*(x+su)-\Phi^*(x+u)}s^{d-1}\,ds \Big)\,
|u|\,|Sg(x+u)|\,dm(u).
\end{align*}

$\Phi^*$ is a concave function for all sufficiently large $\lambda$, 
and \[ x+u = s^{-1}(x+su) + (1-s^{-1})x,\]
so for all $s\in[1,\infty)$ and $x,u\in\reals^d$,
\begin{equation*} \Phi^*(x+u) \ge s^{-1}\Phi^*(x+su) + (1-s^{-1})\Phi^*(x).\end{equation*}
Therefore
\begin{align*} \Phi^*(x)+2\Phi^*(x+su)-\Phi^*(x+u)
&\le (1-s^{-1})\Phi^*(x)+(1+s^{-1})\Phi^*(x+su)
\\& \le \Phi^*(x+su)  \end{align*}
since $\Phi^*\le 0$.
Inserting this bound into the last inequality of the preceding paragraph completes the proof of Lemma~\ref{lemma:gSgbound}.
\end{proof}

\begin{lemma}
\begin{equation}
|u| I(x,u) \le
\begin{cases}
C(1+|x-x^\dagger|)^{d-1}|u|^{1-d} \qquad &\text{for all $u,x$}
\\
Ce^{-c|u|^2} &\text{if } |u|\ge 2|x-x^\dagger|.
\end{cases}
\end{equation}
\end{lemma}

\begin{proof}
Recall that $\Phi^* = \Phi-M_\Phi+\gamma$ is real-valued, nonpositive, and concave, and vanishes at $x^\dagger$.
$\Phi$ has a negative definite Hessian which is uniformly comparable to $-\lambda$,
while $\gamma$ is independent of $\lambda$ and has a Hessian which is bounded above and below.
Therefore 
\begin{equation} \Phi^*(x)\le -c\lambda|x-x^\dagger|^2,\end{equation}
uniformly in $x,\lambda,\xi$ for all sufficiently large $\lambda$.

Therefore for $u\ne 0$,
\begin{align*}|I(x,u)| &\le \int_0^\infty e^{-c\lambda|x-x^\dagger+su|^2}s^{d-1}\,ds 
\\&=|u|^{-d}\int_0^\infty e^{-c|x-x^\dagger-ru|^2/|u|^2}r^{d-1}\,dr
\\&\le  C|u|^{-d}(1+|x-x^\dagger|)^{d-1}.
\end{align*} 

If $|u|\ge 2|x-x^\dagger|$ then for all $s\ge 1$, 
\begin{equation} \Phi^*(x+su)\le -c\lambda|x+su-x^\dagger|^2 \le -c'\lambda s^2|u|^2.\end{equation}
It has already been noted above that
$\Phi^*(x)+2\Phi^*(x+su)-\Phi^*(x+u)\le \Phi^*(x+su)$. 
Consequently 
\begin{align*} 
|I(x,u)| &\le  \int_1^\infty e^{\Phi^*(x+su)}s^{d-1}\,ds
\\&\le e^{-c''|u|^2}\  \text{ if } |u|\ge 2|x-x^\dagger|.\end{align*}
\end{proof}

Therefore there exist $C,c\in\reals^+$ such that
\begin{multline}
|g(x)|\le C(1+|x-x^\dagger|)^{d-1} \lambda^{d/2} \int_{|u|\le 2|x-x^\dagger|} |u|^{1-d} |Sg(x+u)|\,dm(u)
\\+ C \lambda^{d/2} \int_{\reals^d} e^{-c|u|^2} |Sg(x+u)|\,dm(u).
\end{multline}

The second term on the right-hand side represents the action on $|Sg|$ of a bounded linear operator
form $L^2(\reals^d)$ to $L^2(\reals^d)$, whose operator norm is proportional to $\lambda^{d/2}$.
Since the function $u\mapsto |u|^{1-d}$ is a positive decreasing function of $|u|$ and satisfies
\[\int_{|u|\le 2|x-x^\dagger|}|u|^{1-d} \le C|x-x^\dagger|,\]
one has
\begin{multline} (1+|x-x^\dagger|)^{d-1} \lambda^{d/2} \int_{|u|\le 2|x-x^\dagger|} |u|^{1-d} |Sg(x+u)|\,dm(u)
\\ \le C\lambda^{d/2}(1+|x-x^\dagger|)^d M(Sg)(x), \end{multline}
where $M$ is the Hardy-Littlewood maximal function.  Now $M$ is bounded on $L^2(\reals^d)$,
while multiplication  by $(1+|x-x^\dagger|)^d$ defines a bounded operator from $L^2(\reals^d)$
to the weighted space $L^2(\reals^d, w)$ with $w(x) = (1+|x-x^\dagger|)^{-2d}$.
This completes the proof of Lemma~\ref{lemma:weightedlowerconjugated}
and hence of Lemma~\ref{lemma:weightedlower}.

\end{document}